\documentclass[11pt,reqno,a4paper]{amsproc}
\usepackage{amssymb}
\usepackage{amsfonts}
\usepackage{amsmath}
\usepackage{graphicx}
\usepackage{color}
\newtheorem{thm}{Theorem}[section]
\newtheorem{lm}{Lemma}[section]
\newtheorem{defn}{Definition}[section]
\newtheorem{pro}{Proposition}[section]
\newtheorem{re}{Remark}[section]

\newtheorem{cla}{Claim}[section]
\usepackage{esint}

\DeclareMathOperator{\dist}{dist}
\DeclareMathOperator{\ric}{Ric}
\DeclareMathOperator{\Area}{Area}
\DeclareMathOperator{\ADM}{ADM}
\DeclareMathOperator{\Haw}{H}
\DeclareMathOperator{\SH}{\Sigma_{\scriptscriptstyle H}}
\def\SO{{\Sigma_{\scriptscriptstyle O}}}
\def\SH{{\Sigma_{\scriptscriptstyle H}}}
\def\SH{{\Sigma_{\scriptscriptstyle H}}}
\def\ADMm{\mathfrak{m}_{\scriptscriptstyle\ADM}}
\def\Hm{\mathfrak{m}_{\scriptscriptstyle\Haw}}

\begin{document}
\title[]
{On the Rigidity of Riemannian-Penrose Inequality for Asymptotically Flat 3-manifolds with Corners}

 \author{Yuguang Shi}
\address[Yuguang Shi]{Key Laboratory of Pure and Applied Mathematics, School of Mathematical Sciences, Peking University, Beijing, 100871, P.\ R.\ China}
\email{ygshi@math.pku.edu.cn}
\thanks{$^1$Research partially supported by NSFC 11671015}

\author{Wenlong Wang}
\address[Wenlong Wang]{Key Laboratory of Pure and Applied Mathematics, School of Mathematical Sciences, Peking University, Beijing, 100871, P.\ R.\ China}
 \email{wangwl@math.pku.edu.cn}

\author{Haobin Yu}
\address[Haobin Yu]{Key Laboratory of Pure and Applied Mathematics, School of Mathematical Sciences, Peking University, Beijing, 100871, P.\ R.\ China}
 \email{robin1055@126.com}

\renewcommand{\subjclassname}{
  \textup{2010} Mathematics Subject Classification}

\date{July, 2017}

\begin{abstract}
In this paper we prove a rigidity result for the equality case of the Penrose inequality on $3$-dimensional asymptotically flat manifolds with nonnegative scalar curvature and corners. Our result also has deep connections with the equality cases of Theorem 1 in \cite{Miao2} and Theorem 1.1 in \cite{LM}.
\end{abstract}

\maketitle
\markboth{Shi Yuguang, Wang Wenlong and Yu Haobin}{Rigidity of RPI}

\section{Introduction}
In this paper, we are interested in what happens when the equality holds in the Penrose inequality on asymptotically flat manifolds with corners (see Theorem \ref{RPI for mfds with corner} below). This problem has some deep connections with the rigidity of compact manifolds with nonnegative scalar curvature and nonempty boundaries.  Let $M$ be an oriented $n$-dimensional smooth manifold with inner boundary $\SH$. We assume that there exists a bounded domain $\Omega\subset M$ with $\partial\Omega=\Sigma_{\scriptscriptstyle H}\cup\SO$, and
$\SO$ is a smooth hypersurface in $M$. To state precisely, we first review some notions (see \cite{ADM,HI, MM,Miao1}).
\begin{defn}
{\bf A metric $g$ admitting corners along $\Sigma$}
is defined to be a pair of $(g_{-}, g_{+})$, where $g_{-}$ and
$g_{+}$ are smooth metrics on $\Omega$ and
$M \setminus \bar{\Omega}$
so that they are $C^2$ up to the boundary and they induce the same metric on
$\Sigma$.
\end{defn}

\begin{defn} Given $g =(g_{-}, g_{+})$, we say $g$ is
{\bf asymptotically flat} if the manifold
$(M\setminus \Omega, g_{+})$ is asymptotically
flat (AF) in the usual sense, i.e.
 if there is a compact subset $K$ such that $g_+$ is smooth on  $M\setminus K$, and $M\setminus K$ has finitely many components $E_k$, $1\le k\le l$,  each $E_k$ is called an end of $M$, such that each $E_k$ is diffeomorphic to $\mathbb{R}^n\setminus B(R_k)$ for some Euclidean ball $B(R_k)$, and  the followings are true:
In the standard coordinates $x^i$ of $\mathbb{R}^n$,
 \begin{equation}
{g_+}_{ij}=\delta_{ij}+\sigma_{ij}
\end{equation}
with
\begin{equation} \label{daf2}
\sup_{E_k} \left\{\sum_{s=0}^2|x|^{\tau+s}|\partial^s\sigma_{ij}|\right\}<\infty
\end{equation}
for $\tau>\frac{n-2}2$, where $\partial f$ and $\partial^2f$ are the gradient and Hessian of $f$ with respect to the Euclidean metric.
\end{defn}

If the scalar curvature of $g_+$ is $L^1$-integrable on $M\setminus\Omega$, then we can define  the ADM mass as the following:
\begin{defn}
The {\bf Arnowitt-Deser-Misner (ADM) mass}  of an end $E$ of an AF  manifold $M$ is defined as:
\begin{equation} \label{defadm1}
\ADMm(E)=\lim_{r\to\infty}\frac{1}{{ 2(n-1)}\omega_{n-1}}\int_{S_r}
\left(g_{ij,i}-g_{ii,j}\right)\nu^jd\Sigma_r^0,
\end{equation}
\end{defn}

\begin{defn} The {\bf mass} of $g=(g_{-}, g_{+})$ is defined
to be the ADM mass of $g_{+}$ whenever it exits.
\end{defn}
\begin{defn} 
Let $\Sigma$ be a hypersurface of an AF manifold (may have corners). We say $\tilde\Sigma$ is a {\bf minimizing exclosure} of $\Sigma$ if it minimizes area among all hypersurfaces in $M$ that
enclose $\Sigma$. We say $\Sigma$ is {\bf (strictly) outer minimizing} if it is a (the unique) minimizing exclosure of itself. 
\end{defn}

In \cite{MM}, Mccormick and Miao established the following
Riemannian Penrose inequality on asymptotically flat manifolds with corners along
a hypersurface.
\begin{thm}\label{RPI for mfds with corner}
Let $M^n$ denote a noncompact differentiable manifold of dimension
$3\leq n\leq 7$, with compact boundary $\SH$. Let $\SO$ be an embedded hypersurface in the
interior of $M^n$ such that $\SO$ and $\SH$ bounds a bounded domain $\Omega$. Suppose $g$ is a $C^0$
metric on $M^n$ satisfying:\\
$\bullet$ $g$ is smooth on both $M^n\setminus\Omega$ and $\bar{\Omega}=\Omega\cup\SO\cup\SH$;\\
$\bullet$ $(M^n\setminus\Omega, g)$ is asymptotically flat;\\
$\bullet$ $g$ has nonnegative scalar curvature away from $\SO$;\\
$\bullet$ $H_{-}\geq H_{+}$, where $H_{-}$ and $H_{+}$  denote the mean curvature of $\SO$ in $(\Omega, g)$ and
$(M^n\setminus\Omega, g)$, respectively, with respect to the infinity-pointing normal;\\
$\bullet$ $\SH$ is a minimal hypersurface in $(\Omega, g)$ and $\SH$ is outer minimizing in $(M^n, g)$.\\
Then the Riemannian Penrose inequality holds on $(M^n, g)$, i.e.
\[
\mathfrak{m}\geq\frac{1}{2}\left(\frac{|\SH|}{\omega_{n-1}}\right)^{\frac{n-2}{n-1}}
\]
where $\mathfrak{m}$ is the ADM mass of $(M^n\setminus\Omega, g)$.
\end{thm}
In this paper, we study what happens when above inequality holds in $3$ dimension.

We first introduce the following notion.
\begin{defn}
Let $\Sigma\subset\bar{\Omega}$ be a closed minimal surface. $\Sigma$  is called {\bf strictly stable}, if
for any nonzero test function $f\in W^{1,2}(\Sigma)$, it holds
\[
\int_{\Sigma}f\mathcal{L}f>0,
\]
where $\mathcal{L}$ is the stability operator given by $\mathcal{L}=-\Delta-(\ric(\nu,\nu)+|A|^2)$. Here, $\Delta$ is the Laplacian operator with respect to the induced metric on $\Sigma$, $\nu$ and $A$ are  the outward unit normal vector and the second fundamental form of $\Sigma$ in $\Omega$ respectively.
\end{defn}

Our main result is stated as following.
\begin{thm}\label{rigidity}
Let $g=(g_-,g_+)$ be an asymptotically flat metric on a three manifold $M$ satisfying the assumptions as in
Theorem \ref{RPI for mfds with corner}.
Suppose that $\SH$ is strictly stable, as well as strictly outer minimizing. And assume $$
\ADMm\left(g_+\right)=\Hm(\SH, g_-)=\mathfrak{m}.$$
Then we have the following rigidity conclusions: {\rm (i)} $H_-\equiv H_+$ on $\SO$. {\rm(ii)} $(\Omega, g_-)$ is static with vanishing scalar curvature. {\rm(iii)} If in addition, for any $p$ in $(\Omega, g_-)$, its minimal geodesic line to $\SH$ is contained in $\Omega$, and its distance function to $\SH$ is smooth, then $(\Omega, g_-)$ is Schwarzschild, i.e. there is an isometric embedding $F $: $(\Omega, g_-) \mapsto (\mathbb{X}_{\mathfrak{m}}, g_{\scriptscriptstyle{S}})$ that maps $\SH$ to the horizon of $(\mathbb{X}_{\mathfrak{m}}, g_{\scriptscriptstyle{S}})$, where $(\mathbb{X}_{\mathfrak{m}}, g_{\scriptscriptstyle{S}})$ is the Schwarzschild manifold with mass $\mathfrak{m}$.
\end{thm}
\begin{re} In conclusion {\rm(iii)}, the requirement for the shape of $\Omega$ can be viewed as a star-shaped condition in a sense. It might be stronger than needed. 
\end{re}

Theorem \ref{rigidity} also partially answers what happen if the equalities  of   inequalities in Theorem 1 in \cite{Miao2} and Theorem 1.1 in \cite{LM} hold, i.e. if $(\Omega, g_-)$ satisfies Theorem \ref{rigidity}, and corresponding equalities of Theorem 1 in \cite{Miao1} and Theorem 1.1 in \cite{LM} hold, then $(\Omega, g_-)$ is isometric to a domain of a Scharzschild manifold, and $\SH$ is the horizon of this Scharzschild manifold (here for simplicity, we regard the isometric isomorphism as the identity map).

Two key ingredients of our argument are making conformal deformations and using the inverse mean curvature flow (IMCF) on AF manifolds established in \cite{HI}. However, $\SH$, the inner boundary of $(\Omega, g_-)$ being ``horizon'' may not be preserved under global conformal deformations. 

To overcome this difficulty, we construct  a foliation  consisting of constant mean curvature (CMC) surfaces near $\SH$ (see Proposition \ref{foliation}). This foliation possesses some interesting properties that fit IMCF well and turn out to be crucial in the proof. One of these properties is that each leaf is outer minimizing with respect to metrics under small perturbations (not only restricted to conformal deformations,  see Proposition \ref{foliationminimizinghullsurf}  and Lemma \ref{minimizing hull} below). Another is that the mean curvature of each leaf is positive and this property is preserved under small conformal deformations. Furthermore, this foliation is also an isoperimetric foliation so that the Hawking mass is non-decreasing along it (see \cite{Bray}). In particular, the Hawking mass of any leaf is not less than the Hawking mass of $\SH$. 

If $H_+$ and $H_-$ are not equal at some point on $\SO$, or the scalar curvature is positive at some point in $\Omega$, we fix a leaf $\Sigma_1$ very near $\SH$ and do suitable conformal deformations outside $\Sigma_1$. The resulting metrics are smooth AF metrics on the exterior region outside $\Sigma_1$ with nonnegative scalar curvature, and $\Sigma_1$ is still outer minimizing. Then the ADM mass of $M$ is not less than the Hawking mass of $\SH$ with respect to the resulting metrics, due to the arguments in \cite{HI}. But on the other hand, these conformal deformations strictly decrease the ``energy'' of the exterior region, so the AMD mass of the exterior region with respect to the deformed metrics is strictly less than the initial ADM mass, or the Hawking mass of $\Sigma_1$ with respect to the deformed metrics is strictly greater  then the initial Hawking mass. Thus we get a contradiction. If $(\Omega,g_-)$ is not static, by Theorem 1 in \cite{Cor}, there is a small compact perturbation of $g_-$ denoted by $\bar g$, such that $\bar g$ satiefies the assumptions as $g_-$, but has positive scalar curvature at some point in $\Omega$. Then this falls into the case discussed above. Hence, conclusions (ii) and (iii) hold. 

Note that we also have $\Hm(\SH, g_-)=\Hm(\Sigma_1, g_-)$. Then by the results of \cite{Bray} or \cite{HI}, there exists a collar neighbourhood $U$ near $\SH$ in which $g_-$ is the Schwarzschild metric. Finally, when assuming $\Omega$ is ``star-shaped'', conclusion (iii) follows by a uniqueness continuation argument for static metrics established in \cite{CD} (see also \cite{Biq}).

The remains of the paper run as follows: in Section 2, we construct an isoperimetric foliation near $\SH$ and show each leaf close enough to $\SH$ is outer minimizing with respect to certain perturbed metrics; in Section 3, we establish some estimates for certain conformal deformation equations; in Section 4, we complete the proof of Theorem \ref{rigidity}.

\section{Isoperimetric foliation near $\SH$}

In this section, we construct a foliation near $\SH$, of which all leaves are outer minimizing isoperimetric surfaces in $(M,g)$. As a first step, we construct a CMC foliation near $\SH$ via a similar approach in \cite{ACG}.
\begin{pro}\label{foliation}
Let $(M,g)$ and $\SH$ be as stated in Theorem \ref{rigidity}. Then there exists a stable CMC foliation near $\SH$.
\end{pro}

\begin{proof}
In order to construct a CMC surface near $\SH$, we need to extend $\Omega$ to a larger domain. More precisely, we glue $\Omega$ with $(-1, 0]\times \SH $ along $\SH$ in an obvious way. The resulting domain is denoted by $\hat \Omega$. Let $\mathcal{H}(u)$ denote the mean curvature of the hypersurface $\Sigma_u$, which is defined as following
\[
\Sigma_u:x\rightarrow \exp_x\left[u(x)\nu\right],\ u\in C^{\infty}(\SH),
\]
for small enough $u$. Then we see that $\Sigma_u \subset \hat \Omega$ for sufficiently small $u$ and it is contained in $\Omega$ if $u> 0$. We are going to show that there is a family of stable CMC surfaces $\Sigma_u$ with small positive $u$, hence $\Sigma_u \subset \Omega$. By a direct computation, we see that the linearization of $\mathcal{H}$ at $u \equiv0$ is $\mathcal{H}'(0)=\mathcal{L}$.
 We introduce the operator $\mathcal{H}^*$ defined by
\[
\mathcal{H}^*:C^{2, \beta}(\SH)\times\mathbb{R}\rightarrow C^{\beta}(\SH)\times\mathbb{R},
\ \ \mathcal{H}^*(u, k)=\Big(\mathcal{H}(u)-k, \int_{\SH}u\Big),
\]
for some $0<\beta<1$. Suppose that $\mathcal{L}v=s$ for some constant $s$ with $\int_{\SH}v=0$. Then by stability assumption of $\SH$ we see that
$v\equiv0, s=0$. Thus, $\mathcal{H}^*$ has a invertible linearization at $(0,0)$.
Then by the inverse function theorem, for each $\tau$ sufficiently small there exist $u_{\tau}$ and
$k_{\tau}$ depending on $\tau$ smoothly with
\begin{equation}\label{uk}
\mathcal{H}(u_{\tau})=k_{\tau},\ \ \int_{\SH}u_{\tau}=\tau.
\end{equation}
Set $$v(x)=\frac{\partial u_{\tau}(x)}{\partial \tau}\Big |_{\tau=0}, \quad s=\frac{\partial k_{\tau}}{\partial\tau}\Big|_{\tau=0}.$$
We must have
\begin{equation}\label{derivative}
 v(x)>0\,\,\,\,\mbox{and}\,\,\,\,s>0.
\end{equation}
Indeed, a differentiation of $(\ref{uk})$ yields
\[
\mathcal{L}v=s, \ \ \int_{\SH}v=1.
\]
It follows that
\[
s=\int_{\SH}sv=\int_{\SH}v\mathcal{L}v>0.
\]
To show that $v>0$, by the equation $\mathcal{L}v=s>0$ and the maximum principle, it suffices to prove that $v\geq 0$.
If this fails, then we can choose a nonempty connected component $\Gamma$ of the open set $\{x\in\SH|\,v(x)<0\}$.
Set
\begin{equation}
f(x)=\left\{
\begin{aligned}
&v(x) \quad\mbox{for}\,\,x\in\Gamma;\\
&\,\,\,0 \quad\,\,\,\,\,\mbox{for}\,\,x\in\SH\backslash\Gamma.
\end{aligned}
\right.\nonumber
\end{equation}
Then it follows that
\[
0>\int_{\Gamma}sf=\int_{\Gamma}f\mathcal{L}f=\int_{\SH}f\mathcal{L}f,
\]
which contradicts the stability condition.

Note that $u(0)=0$. Then from \eqref{derivative} we see that $u_\tau(x)>0$ and $u_{\tau}(x)$ is strictly increasing with respect to $\tau$ for sufficiently small $\tau$ and all $x\in\SH$. Thus,  the constant mean curvature surfaces $\Sigma_{u_{\tau}}$ are contained in $\Omega$ and disjoint from each other. Since $\SH$ is strictly stable, we see that $\Sigma_{u_{\tau}}$  is also stable for sufficiently small $\tau$.

Finally, we show that for any $p$ near $\SH$ in $\Omega$, there is some $\Sigma_{u_{\tau}}$ passing through. Suppose not, then there exists a largest $\tau_1$ so that $p$ is not in the closure of the domain enclosed by $\Sigma_{u_{\tau_1}}$ and $\SH$, and also exists a smallest $\tau_2$ so that $p$ is in the closure of the domain enclosed by $\Sigma_{u_{\tau_2}}$ and $\SH$, and $\tau_2 > \tau_1$. While this is impossible as $u_\tau$ is strictly increasing with $\tau$. Thus, for  small $\tau$, $\Sigma_{u_{\tau}}$ form a foliation in a collar neighborhood of $\SH$.
\end{proof}

It's not hard to see that the volume of the domain enclosed by $\Sigma_{u_{\tau}}$ and $\SH$ is
given by
\[
V_{\tau}=\int_{\SH}u_{\tau}+O\left(\tau^2\right).
\]
Then we have
\[\frac{\partial V_{\tau}}{\partial \tau}\Big|_{\tau=0}=
\int_{\SH}v>0.
\]
Thus, the foliation obtained in Proposition \ref{foliation}  can be written as $\{\Sigma_V\}_{0\leq V\leq V_0}$ such that
$\Sigma_V$ smoothly depends on $V$, which is the volume of the domain enclosed by $\Sigma_V$ and $\SH$. Next we prove:

\begin{pro}\label{foliationminimizinghullsurf}
There exists some $V_1\in (0,V_0]$ such that each leaf of $\{\Sigma_V\}_{0\leq V\leq V_1}$ is the isoperimetric surface for volume $V$, and is outer minimizing.
\end{pro}

\begin{proof}
For convenience, we use the convention that all surfaces mentioned in this proof enclose $\SH$.
We use $U_0$ to denote the region enclosed by $\SH$ and $\Sigma_{V_0}$. By the result in \cite{Ros}, there exists some positive $\bar V\leq V_0$ such that for any $V\leq\bar V$, there exists an isoperimetric surface $\bar\Sigma_V$ enclosing a region of volume $V$ with $\SH$, which is also contained in $\Omega$ and close to $\SH$. Note that each $\bar\Sigma_V$ has constant mean curvature. By the maximum principle, any surface $S\subset U_0$ with the same constant mean curvature must coincide with some $\Sigma_V$. It follows that $\bar\Sigma_V=\Sigma_V$ for any $V\in [0,\bar V]$. Thus we prove the first part of Proposition \ref{foliationminimizinghullsurf}. 

For the second part, we take a contradiction argument. It is not hard to see from \eqref{derivative} that $\Area(\Sigma_V)$ is strictly monotonically increasing in $[0,V_2]$ for some $0<V_2\leq \bar V$. $\Sigma_{V_2}$ divides $\Omega$ into two parts. We denote the part bounded by $\SH$ and $\Sigma_{V_2}$ by $W_-$, the other by $W_+$. If the proposition is not true, then there exist $V_i\rightarrow 0$ and a sequence of surfaces $S_{V_i}$ such that each $S_{V_i}$ encloses $\Sigma_{V_i}$ with $\Area(S_{V_i})<\Area(\Sigma_{V_i})$. We must have $S_{V_i}\cap W_+\neq \emptyset$. Otherwise, the volume of the region enclosed by $S_{V_i}$ and $\SH$, which is denoted by $\tilde V_i$, is not greater than $V_2$. Since $\Sigma_{\tilde V_i}$ is the isoperimetric surface for volume $\tilde V_i$, $\Area(S_{V_i})\geq\Area({\Sigma_{\tilde V_i}})$. Because $S_{V_i}$ encloses $\Sigma_{V_i}$, $\tilde V_i\geq V_i$. By the monotonicity, we have $\Area(\Sigma_{\tilde V_i})\geq\Area(\Sigma_{V_i})$. Thus we have $\Area(S_{V_i})\geq\Area(\Sigma_{V_i})$, which contradicts with our assumption. Hence, $S_{V_i}\cap W_+\neq \emptyset$. By the result in \cite{HI}, each $S_{V_i}\cap W_+$ is a minimal surface. Note that $H(\Sigma_{V_2})>0$. By the maximum principle, there is a positive constant $\epsilon_0$ such that $\sup_{y\in S_{V_i}\cap W_+}\dist(y,\Sigma_{V_2})>\epsilon_0$ for all $i$. It is not hard to see that the volume of the region enclosed by $\SH$ and $S_{V_i}$ is uniformly bounded. Then in the sense of current (see Theorem 32.2 in \cite{Sim}), $S_{V_i}$  weakly converge to $S_{\infty}$,  which is an integer multiplicity current, encloses $\SH$ and has $2$-dimensional Hausdorff measure not larger than the area of $\SH$. Since $\SH$ is strictly outer minimizing, $S_{\infty}$ should coincide with $\SH$ except a zero measure set. On the other hand, by the estimate in \cite{MY}, we have $\Area(S_{V_i}\cap W_+)>\epsilon_1$ for a certain positive constant $\epsilon_1$ independent of $i$. As a consequence, we also have  $\Area(S_{\infty}\cap W_+)\geq\epsilon_1$. Thus we reach the contradiction. Hence, any leaf of $\{\Sigma_V\}_{0\leq V\leq V_1}$ should be outer minimizing.
\end{proof}
We let $\mathcal{F}$ denote this foliation, $W_0$ the corresponding region. By the result in \cite{Bray}, the Hawking mass of $\Sigma_t$ is monotonically increasing for $t\in[0,V_1]$.
\section{Some estimates for conformal deformation equations}
In this section, we establish some estimates for certain conformal deformation equations. 

Set $\Sigma_-^s=\{\exp_x(-s\nu_{g_-}), x\in\SO\}$, $\Sigma_+^s=\{\exp_x(s\nu_{g_+}), x\in\SO\}$, where $s\in(0,s_0]$ and $s_0$ is a fixed positive constant.
In \cite{Miao1}, Miao constructed a family of $C^2$ metrics $\{g_{\delta}\}_{0<\delta\leq\delta_0}$
on $M$ with
\begin{equation}
g_{\delta}=\left\{
\begin{aligned}
dt^2+\sigma_{\delta ij}(x,t)dx^idx^j,\ \ (x,t)\in \SO\times(-s_0, s_0)\\
g,\ \ \ \ \ \ \ \ \ \ \ \ \ \ \ \ \ \ \ (x,t)\notin \SO\times(-s_0,s_0)
\end{aligned}
\right.
\end{equation}

Let $\phi(t)\in C^{\infty}_c([-1, 1])$ be a standard mollifier such that
\[
0\leq \phi\leq 1,\ \ \phi\equiv 1 \ \  \text{in} \ \ [-\frac{1}{3},\frac{1}{3}] \ \ \text{and}\ \ \int_{-1}^1\phi(t)dt=1.
\]
The properties of $\{g_{\delta}\}_{0<\delta\leq\delta_0}$ is given in  the following proposition which can be found in \cite{Miao1}.
\begin{pro}\label{smoothing the metric g}
Let $g=(g_{-},g_{+})$ be a metric admitting corners along $\SO$.
Then there exists a family of $C^2$ metrics $\{g_{\delta}\}_{0<\delta\leq\delta_0}$ on $M$
 so that $g_{\delta}$ is uniformly close to $g$ on $M$, $g_{\delta}=g$
  outside $\SO\times(-\frac{\delta}{2}, \frac{\delta}{2})$
   and the scalar curvature of $g_{\delta}$ satisfies
\begin{align}
R_{\delta}(x, t)=& O(1),
\  \text{for}\  (x, t)\in \SO\times\left\{\frac{\delta^2}{100}<t\leq\frac{\delta}{2}\right\}\label{RO(1))}\\
R_{\delta}(x, t)= &O(1)+\left\{H(\SO,g_{-})(x)-H(\SO,g_{+})(x)\right\}
\left\{\frac{100}{\delta^2}\phi\left(\frac{100t}{\delta^2}\right)\right\},\nonumber\\
&\text{for}\ (x, t)\in \SO\times\left[-\frac{\delta^2}{100}, \frac{\delta^2}{100}\right]\label{RHH}
\end{align}
where $O(1)$ represents quantities that are bounded by constants depending
only on $\mathcal{G}$, but not on $\delta$.
\end{pro}

Fix some $\Sigma_1\in\mathcal{F}$ that is close to $\SH$ and denote the region outside of $\Sigma_1$ by $M_1$. As in \cite{Miao1}, we need to do a conformal deformation to get a smooth asymptotically flat metric with nonnegative  scalar curvature. Set $R_{\delta-}=\max\{-R_\delta, 0\}$. Consider the following equation
\begin{equation}\label{conformal equation1}
\left\{
\begin{aligned}
\Delta_{g_{\delta}}u_{\delta}+\frac{1}{8}R_{\delta-}u_{\delta}&=0\ \ \,\text{in}\ M_1\\
u_{\delta}&=1\ \ \, \text{on}\ \Sigma_1\\
u_\delta&\rightarrow 1 \ \ \text{at}\  \infty.
\end{aligned}
\right.
\end{equation}

Due to \eqref{RO(1))} and \eqref{RHH}, for sufficiently small $\delta$, equation \eqref{conformal equation1} has a $C^2$ positive solution $u_{\delta}$. The maximum principle implies that $ u_{\delta}\geq 1$.
Then by Proposition 4.1 in \cite{Miao1}, we have
$u_{\delta}\rightarrow 1$ in the $C^2_{loc}$ sense in $M_1\setminus\SO$.
The resulting conformal metrics are denoted by $\tilde{g}_{\delta}=u_{\delta}^4g_{\delta}$.

For further use, we focus on the following two cases.

\textbf{Case 1.} $H_{-}(p)-H_{+}(p)>0\ \ \text{for some}\  p\in\SO.$

In this case, by (\ref{RHH}), the scalar curvature of $\tilde{g}_{\delta}$ is positive in a small neighbourhood of $p$.
We want to use another conformal deformation to decrease the scalar curvature of $\tilde{g}_{\delta}$ a bit.
Set
\[
B_r(p)=\{x\in M:d_g(x,p)<r\}.
\]
We define a two-parameters family of functions by $\psi_{\delta,r}=\eta_rR(\tilde g_{\delta})$ with
 $\eta_r\in C^\infty_0(B_{r}(p))$ being a family of cutoff functions satisfying:
\begin{equation}
\eta_r(x)=\left\{
\begin{aligned}
&\frac{1}{8}, \quad\mbox{for}\,\,x\in B_{\frac{r}{2}};\\
&0, \quad\mbox{for}\,\,x\in B_{r}(p)\setminus B_{\frac{3r}{4}};\\
&0\leq \eta_r(x)\leq \frac{1}{8},\quad\mbox{for}\,\, x\in B_{\frac{3r}{4}}\setminus B_{\frac{r}{2}}(p).
\end{aligned}
\right.\nonumber
\end{equation}
Then
for any $\delta, r$ with $\delta\leq r\ll 1$, by Proposition \ref{smoothing the metric g}, we have
\begin{equation}\label{integral estimate}
 \varepsilon_0 r^2\leq \int_{M}\psi_{\delta,r}dV_{\tilde{g}_{\delta}}\leq \frac{1}{\varepsilon_0}r^2,
\end{equation}
where $\varepsilon_0$ is a constant depending only on $g$.

Now we consider the following equation
\begin{equation}\label{conformal equation2}
\left\{
\begin{aligned}
{\Delta}_{\tilde{g}_{\delta}}v_{\delta,r}-\psi_{\delta,r}v_{\delta,r}&=0\ \ \, \text{in}\ M_1\\
v_{\delta,r}&=1\ \ \, \text{on}\ \Sigma_1\\
v_{\delta,r}&\rightarrow 1\ \ \text{at}\ \infty.
\end{aligned}
\right.
\end{equation}
By the maximum principle, $0<v_{\delta,r}\leq 1$ in $M$.
Then for any  $\delta$ small enough, we have
\begin{align}
\int_{M}\left|\nabla_{\tilde g_\delta}v_{\delta,r}\right|^2dV_{\tilde{g}_{\delta}}
=&\int_{M}\psi_{i,j}v_{\delta,r}(1-v_{\delta,r})dV_{\tilde{g}_{\delta}}\nonumber\\
\leq&\frac{1}{\varepsilon_0}r^2.\label{L2 estimate}
\end{align}
On the other hand, given $K\subset\subset \bar{M}_1\setminus \SO$, by the standard Schauder estimates for linear elliptic equations,
 for any $\delta\leq r\ll 1$ small enough, we have
\begin{equation}\label{C2 estimate}
\| v_{\delta,r}\|_{C^2(K)}\leq C_K,
\end{equation}
where the norm $\|\cdot\|$ is taken with respect to $\tilde{g}_{\delta}$, and
$C_K$ is a constant depending only on $K$ and the initial metric $g$.

Set $\hat{g}_{\delta,r}=v_{\delta,r}^4\tilde{g}_{\delta}$.
Let $\Sigma_2$ and $\Sigma_3$ be two leaves in $\mathcal{F}$ so that $\Sigma_3$ encloses $\Sigma_2$ and $\Sigma_2$ encloses $\Sigma_1$.
In the rest of the paper, we always use $\Omega_i$ to denote the domain enclosed by $\Sigma_i$ and $\Sigma_1$, for $i= 2, 3$.
We have the following
\begin{lm}\label{deform1}
There exists some $r_0\ll1$ such for any $\delta, r$ with $\delta\leq r\leq r_0$,
it holds
$$H\left(\Sigma', \hat{g}_{\delta,r}\right)>0,
$$
where $\Sigma'$ is any leaf  in $\bar{\Omega}_3$.
\end{lm}

\begin{proof}
Denote by $H_1$ the mean curvature of $\Sigma_1$ with respect to $g_-$. Note that $v_{\delta, r}\equiv1$ on $\Sigma_1$. Combining (\ref{L2 estimate}) and (\ref{C2 estimate}) we see that
 there exists some $r_0 \ll1 $ such that for any $\delta\leq r\leq r_0$,
\begin{equation}\label{C1 estimate}
\|v_{\delta,r}-1\|_{C^1(\bar\Omega_3)}<\frac{H_1}{10}.
\end{equation}
Let $\nu$ be the unit out normal vector filed on $\Sigma'$ with respect to $g_-$.
 Then a direct computation shows
\begin{align*}
H\left(\Sigma', \hat{g}_{\delta,r}\right)
=&v_{\delta,r}^{-2}H\left(\Sigma',\tilde{g}_{\delta}\right)+4u_{\delta}^{-2}v_{\delta,r}^{-3}\frac{\partial v_{\delta,r}}{\partial\nu}\nonumber\\
=&v_{\delta,r}^{-2}\left(u_{\delta}^{-2}H\left(\Sigma',{g}_{\delta}\right)+4u_{\delta}^{-3}\frac{\partial u_{\delta}}{\partial\nu}\right)+4v_{\delta,r}^{-3}\frac{\partial v_{\delta,r}}{\partial\nu}\nonumber\\
>&0\label{H>0}.
\end{align*}\end{proof}
Now we are ready to prove
\begin{lm}\label{minimizing hull}
Let $\Sigma_1$ be a leaf in $\mathcal{F}$ which is close enough to $\SH$.
 Then there exist some $0<\varepsilon_1<r_0$ such that for any $\delta, r$ with $\delta\leq r\leq \varepsilon_1$,
 $\Sigma_1$ is still outer minimizing in $(M_1, \hat{g}_{\delta, r})$.
\end{lm}

\begin{proof}
Suppose  Lemma \ref{minimizing hull} fails, then there exist a sequence of $\delta_k\leq r_k$ with
$r_k \rightarrow0$, and a sequence of surfaces $\hat{\Sigma}_k$ satisfying
\[
\Area(\hat{\Sigma}_k, \hat{g}_{\delta_k, r_k})<\Area(\Sigma_1, \hat{g}_{\delta_k, r_k}).
\]
In this case, $\hat{\Sigma}_k \setminus \Sigma_1\neq\emptyset$
and each $\hat{\Sigma}_{k}\setminus \Sigma_1$ is a minimal surface with respect to $\hat{g}_{\delta_k, r_k}$.

To get the desired contradiction, we use an explicit ``cut and paste" argument by replacing some pieces of $\hat{\Sigma}_k$ with suitable open surfaces whose area can be easily estimated. We use $\hat{g}_k$ to denote $g_{\epsilon_k, \delta_k}$  for short. Let $R$ be a large constant to be determined.

\begin{cla}\label{bound1}
$\hat{\Sigma}_{k}\subset B_R(p)$ for any $k$ large enough.
\end{cla}

Suppose not, then by the assumption that $\hat{g}_{k}$ and $g_-$ induce the same metric on $\Sigma_1$, we have
\[
\Area(\hat{\Sigma}_{k},\hat{g}_k)\leq \Area(\Sigma_1, \hat{g}_k)=\Area(\Sigma_1, g).
\]

For $k$ large enough,  by the asymptotical flatness of $g_+$,  it's easy to show that
 $|\text{Sec}(\hat{g}_{\delta_k,r_k})|(x)\leq 1$,  for any  $x\in B_R(p)\setminus B_{R/2}(p)$ provided $R$ is large enough.
Then by Lemma 1 in \cite{MY},
\[
\Area(\hat{\Sigma}_{k},\hat g_{k})>CR^2\geq2\Area(\Sigma_1, g)>\Area(\Sigma_1, \hat g_{k}).
\]
Thus, we get the desired contradiction when $R$ is large enough.

Set $\Sigma_k'=\hat{\Sigma}_{_k}\cap\bar{\Omega}_2,\ \Sigma_k''=\hat{\Sigma}_{k}\setminus\Omega_2$ for $k\gg 1$.
We use $U_k$ to denote  the domain enclosed by $\hat{\Sigma}_{k}$ and $\Sigma_1$  and set
\[
\Sigma_k^2=\partial(U_k-\Omega_2)-\Sigma_k'', \ \,\,\text{for}\,\,k\gg 1.
\]

For any given $0<\eta\ll1$,  we only need to consider the case that $\hat{\Sigma}_k\cap B_{\eta}(p)\neq\emptyset$.
Let $\hat{\Sigma}_k''$ the the surface that
\[
\hat{\Sigma}_k''\setminus\Sigma_k''\subset S_{\eta}(p), \ 
\hat{\Sigma}_k''\cap B_{\eta}(p)=\emptyset,\ 
\text{and}\  \partial(\hat{\Sigma}_k^{''}\cap  \bar{B}_{\eta}(p))=\partial(\Sigma_k''\cap  \bar{B}_{\eta}(p)).
\]
By choosing $k$ large enough, we have
\begin{equation}
\begin{split}\label{transfer}
\Area(\hat{\Sigma}_{k}, \hat{g}_{k})=&\Area(\Sigma_k', \hat{g}_{k})+\Area(\Sigma_k'', \hat{g}_{k})\\
\geq&\Area(\Sigma_k', g)+\Area(\hat{\Sigma}_k'', g)-\eta\\
\geq&\Area(\Sigma_k', g)+\Area(\Sigma_k^2, g)-\eta.
\end{split}\end{equation}
\begin{cla}\label{areaineq}
There holds
\begin{equation}\label{excessieq}
\Area\left(\Sigma_k', g\right)+\Area\left(\Sigma_k^2, g\right)\geq\Area\left(\Sigma_1, g\right)+\kappa\Area\left(\Sigma_k^2, g\right),
\end{equation}
where $\kappa$ is a positive constant independent of $k$.
\end{cla}

Indeed, as mentioned  above, $\Sigma_2 $ and $\Sigma_3$ can be regarded as graphs on $\Sigma_1$ by the nearest point projection. Define $\tilde\Sigma^2_k$ to be the set of points on $\Sigma_1$ corresponding to that of $\Sigma^2_k$. Since the volume element is increasing everywhere from $\Sigma_1$ to $\Sigma_2$, we have
\begin{equation*}
(1-\kappa)\Area\left(\Sigma_k^2, g\right)\geq\Area\left(\tilde\Sigma_k^2, g\right)
\end{equation*}
for a certain positive constant $\kappa$ independent of $k$. It is obvious that
$$\Area\left(\tilde\Sigma_k^2, g\right)\geq\Area\left(\tilde\Sigma_k^2\cap\bar U_k, g\right).$$
Thus, if we prove
\begin{equation}\label{alternative}
\Area\left(\Sigma'_k\backslash\Sigma_1,g\right)\geq\Area\left((\Sigma_1\backslash\tilde\Sigma^2_k)\cap\bar U_k,g\right),
\end{equation}
we reach our goal. This can be done by using the divergence theorem.

Let $\Omega^2_k$ be the region  in $\Omega_2$ consisting of geodesic jointing the points in $\Sigma^2_k$ with their nearest point projection in $\tilde \Sigma^2_k$. Consider a function $d_1(x)$ defined on $M$ which is  the distance to $\Sigma_1$ with respect to metric $g$. For points on $\Sigma_1$, we have $\nabla_g d_1=\nu_g$ where $\nu_g$ is the outward unit normal vector field of $\Sigma_1$ with respect to $g$.  And we have $\Delta_gd_1|_{\Sigma_1}=H_1>0$.
 Without loss of generality,
 we may require that $\Sigma_2$ and $\Sigma_3$ are sufficiently close to $\Sigma_1$ such that $\Delta_g d_1>0$ in $\bar\Omega_3.$

We use divergence theorem for $\nabla_g d_1$ in $(\Omega_2\backslash\Omega^2_k)\cap\bar U_k$. Namely, we have
\begin{equation*}
\oint_{\partial((\Omega_2\backslash\Omega^2_k)\cap\bar U_k)}\frac{\partial d_1}{\partial\nu_g}\,dS_g=\int_{(\Omega_2\backslash\Omega^2_k)\cap\bar U_k}\Delta_g d_1\,dV_g>0.
\end{equation*}
By expanding the boundary integral on the left hand, we obtain
\begin{equation*}
\int_{\Sigma'_k\backslash\Sigma_1}\frac{\partial d_1}{\partial\nu_g}\,dS_g+\int_{(\Sigma_1\backslash\tilde\Sigma^2_k)\cap\bar U_k}\frac{\partial d_1}{\partial\nu_g}\,dS_g+\int_{(\partial\Omega^2_k\backslash(\Sigma^2_k\cup\tilde\Sigma^2_k))\cap\bar U_k}\frac{\partial d_1}{\partial\nu_g}\,dS_g>0.
\end{equation*}
Since $|\nabla_gd_1|=1$, we have $$\Area\left(\Sigma'_k\backslash\Sigma_1,g\right)\geq\int_{\Sigma'_k\backslash\Sigma_1}\frac{\partial d_1}{\partial\nu_g}\,dS_g.$$
On $\Sigma_1$, $\nabla_gd_1=-\nu_g$. On $\partial\Omega^2_k\backslash(\Sigma^2_k\cup\tilde\Sigma^2_k)$, $\nabla_gd_1\perp\nu_g$. Then it is not hard to see \eqref{alternative} holds.

By assumption we have
\begin{equation}\label{assum}
\Area(\hat{\Sigma}_{k}, \hat{g}_{k})\leq\Area(\Sigma_1, \hat{g}_{k})=\Area(\Sigma_1, g).
\end{equation}
Combination of \eqref{transfer}, \eqref{excessieq} and \eqref{assum} gives
\[
\kappa\Area\left(\Sigma_k^2, g\right)<\eta,\ \ k\gg1.
\]
Consequently we have
\begin{equation}\label{lambdarea}
\Area\left(\Sigma_k^2, g\right)\rightarrow 0\,\,\,  \text{as}\,\, k\rightarrow\infty.
\end{equation}

\begin{cla}
 For $k$ large enough, we have $\hat{\Sigma}_k\subset\bar{\Omega}_3$.
\end{cla}

Suppose the claim is not true. Note that $\hat{\Sigma}_{k}\setminus\Omega_2$ is a minimal surface.
Then by Lemma 1 in \cite{MY}, we have
\begin{equation}\label{lowerbound}
\Area\left(\Sigma_k'', \hat{g}_k\right)>\delta_1
\end{equation}
for some $\delta_1$ independent of $k$. Then choose $\eta\ll\frac{\delta_1}{2}$ and sufficiently large $k$, we have
\begin{align*}
\Area\left(\hat{\Sigma}_{k}, \hat{g}_{k}\right)=&\Area(\Sigma_k', \hat{g}_{k})+\Area(\Sigma_k'', \hat{g}_{k})\\
\geq&\Area(\Sigma_k', g)-\eta+\delta_1\\
\geq&\Area\left(\Sigma_1, g\right)-\Area\left(\Sigma_k^2, g\right)-\eta+\delta_1\\
=&\Area\left(\Sigma_1, \hat g_k\right)-\Area\left(\Sigma_k^2, g\right)-\eta+\delta_1\\
>&\Area(\Sigma_k, \hat g_k)+\frac{\delta_1}{2}.
\end{align*}
Hence, we see $\hat{\Sigma}_{k}$
 is not a minimizing exclosure of $\Sigma_1$ with respect to $\hat{g}_{k}$ for sufficiently large $k$, which contradicts our assumption. So the claim is true.

Now for large enough $k$, we have $\hat{\Sigma}_{k}\subset\bar{\Omega}_3$ and
$H(\hat{\Sigma}_{k}\setminus\Sigma_1, \hat{g}_{k})\equiv0$.
However, for any leaf $\Sigma'\subset\bar{\Omega}_3$, we have $H(\Sigma', \hat{g}_k)>0$ for $k$ large enough.
By maximum principle, we know that $H(\hat{\Sigma}_{k}\setminus\Sigma_1, \hat{g}_{k})>0$ for $k$ large enough.
Hence, we get the desired contradiction.
\end{proof}
\begin{re}\label{outerminimizing in small perturbation}
Denote by $S_{\rho}$ the coordinate sphere $\{(x_1, x_2, x_3) \in M_1: x^2_1 + x^2_2 +x^2_3=\rho^2\}$ near the infinity
and $B_{\rho}$ the region enclosed by $S_\rho$ and $\Sigma_1$.
Let $g_{\varepsilon,R}$ (may admit corner along $\SO$) be the asymptotically flat metrics in $M_1$ with
\[
\|g_{\epsilon, R} - g\|_{C^0(B_R)} \leq \epsilon, \
\|g_{\epsilon, R} - g\|_{C^2(\bar{\Omega}_3)} \leq \epsilon,\
 \text{and}\ \
\|g_{\epsilon, R}- g\|_{C^2 (B_R\setminus B_{R/2})} \leq \epsilon,
\]
where $R$ is a fixed constant large enough. Using the same argument as above, we can show that for sufficiently small $\varepsilon$,
$\Sigma_1$ is outer minimizing in $(M_1,g_{\varepsilon, R})$.
\end{re}

Next, by the similar arguments in the proof of Theorem 5.2 in \cite{ST}, we are able to show

\begin{lm}\label{estimate2}
There exists some $0<\varepsilon_2\leq \varepsilon_1$ such that for
any $\delta, r$ with
$\delta\leq r\leq\varepsilon_2$, we have
\begin{equation}
\int_{M_1}\psi_{\delta, r}(1-v_{\delta, r})dV_{\tilde{g}_{\delta}}
\leq\frac{1}{2}\int_{M_1}\psi_{\delta, r}dV_{\tilde{g}_{\delta}}.
\end{equation}
\end{lm}
\begin{proof}
Let $\theta\ll 1$ be a constant to be fixed.
We use $C$ to denote a uniform constant independent of $\delta, r$ and $\theta$,
 which may vary from line to line.
For $(z,t)\in\SO\times[-\theta, \theta]$, we have
\begin{align}
&\int_{M_1}\psi_{\delta, r}(1-v_{\delta, r})dV_{\tilde{g}_{\delta}}\nonumber\\
=&\int_{B_r(p)}\psi_{\delta, r}(v_{\delta, r}(z,\theta)-v_{\delta, r}(z,t))dV_{\tilde{g}_{\delta}}
+\int_{B_r(p)}\psi_{\delta, r}(1-v_{\delta, r}(z,\theta))dV_{\tilde{g}_{\delta}}.\label{two terms}
\end{align}
To give the estimate of the first term in (\ref{two terms}),
using
\[
\left|v_{\delta,r}(z,\theta)-v_{\delta,r}(z,t)\right|
\leq \left|\int_{t}^{\theta}\frac{\partial}{\partial \tau} v_{\delta, r}(z,\tau)d\tau\right|,
\]
we have
\begin{align}
&\int_{B_r(p)}\psi_{\delta, r}(v_{\delta, r}(z,\theta)-v_{\delta, r}(z,t))dV_{\tilde{g}_{\delta}}\nonumber\\
\leq&\int_{B_r(p)}\psi_{\delta,r}\int_{t}^{\theta}
 |\nabla_{\tilde{g}_{\delta}}v_{\delta, r}|(z,\tau)d\tau dV_{\tilde{g}_{\delta}}\nonumber\\
 \leq&C\int_{B_r(p)\cap\SO\times\left[-\frac{\delta^2}{100}, \frac{\delta^2}{100}\right]}\frac{1}{\delta^2}
 \phi\left(\frac{100t}{\delta^2}\right)\int_{t}^{\theta} |\nabla_{\tilde{g}_{\delta}}v_{\delta, r}|(z,\tau)d\tau dV_{\tilde{g}_{\delta}}\nonumber\\
 \leq&\frac{C}{\delta^2}\int_{-\frac{\delta^2}{100}}^{\frac{\delta^2}{100}}\int_{B_r(p)\cap\SO}
 \int_{t}^{\theta}|\nabla_{\tilde{g}_{\delta}}v_{\delta, r}|(z,\tau)d\tau d\SO dt\nonumber\\
 \leq&C\int_{-\theta}^{\theta}\int_{B_r(p)\cap\SO}
  |\nabla_{\tilde{g}_{\delta}}v_{\delta, r}|(z,t)d\SO dt\nonumber\\
  \leq&C\int_{(\SO\cap B_r(p))\times[-\theta, \theta]}
  |\nabla_{\tilde{g}_{\delta}}v_{\delta, r}|(z,t)dV_{\tilde{g}_{\delta}}\nonumber\\
  \leq&C\left(\int_{M}|\nabla_{\tilde{g}_{\delta}}v_{\delta, r}|^2dV_{\tilde{g}_{\delta}}\right)^{\frac{1}{2}}
  \left(\int_{(\SO\cap B_r(p))\times[-\theta, \theta]}dV_{\tilde{g}_{\delta}}\right)^{\frac{1}{2}}\nonumber\\
  \leq &C\theta r^2.\label{estimate of first term}
\end{align}
Here we have used (\ref{RHH}) in the second inequality and Holder inequality in the penultimate inequality.
Choose $\theta<\frac{\varepsilon_0 }{4C}$, then combining (\ref{integral estimate}) and (\ref{estimate of first term}) gives
\begin{equation}\label{first term}
\int_{B_r(p)}\psi_{\delta, r}(v_{\delta, r}(z,\theta)-v_{\delta, r}(z,t))dV_{\tilde{g}_{\delta}}\leq
\frac{1}{4}\int_{M}\psi_{\delta, r}dV_{\tilde{g}_{\delta}}.
\end{equation}
Once given $\theta$, we may choose $0<\varepsilon_2\leq \varepsilon_1$ such that
\[
0\leq 1-v_{\delta, r}(z,\theta)\leq \frac{1}{4},\ \ \forall z\in\SO.
\]
It follows that
\begin{equation}\label{second term}
\int_{B_r(p)}\psi_{\delta, r}(1-v_{\delta,r})dV_{\tilde{g}_{\delta}}
\leq\frac{1}{4}\int_{B_r(p)}\psi_{\delta, r}dV_{\tilde{g}_{\delta}}.
\end{equation}
Combining (\ref{first term}) and (\ref{second term}) gives the desired result.
\end{proof}

\textbf{Case 2.}
$H_-\equiv H_+$ on $\SO$, but $R(g) >0$ at some point $p \in M_1\setminus \SO$.

In this case, we want to  decrease the scalar curvature a little bit in a neighborhood of a point  by a conformal deformation, and then we want to  investigate the behavior of the solution to the relevant  equation.

For some point $p \in M_1\setminus \SO$, we define a function $\phi\in C^2_0(B_r(p))$ such that $0\leq\phi\leq R_g/2$ with $\phi(p)=\frac{R(p)}{2}$.
Let $f_{\delta, j}$ be a family of $C^2$ functions  defined  in $M_1$ by
\begin{equation}
f_{\delta, j}(x)=\left\{
\begin{aligned}
&\frac{\phi}{2^j}\quad \text{for any } x\in M\setminus \Sigma_0 \times (-\delta, \delta)\\
&-C_1 \quad \text{for any } x\in \Sigma_0 \times [-\frac{\delta}2,\frac{ \delta}2]\\
&-C_1 \leq f_{\delta, j}(x)\leq 0\quad \text{for any } x\in \Sigma_0 \times [-\delta, -\frac{\delta}2]\cup  \Sigma_0 \times [\frac{\delta} 2, \delta].
\end{aligned}
\right.\nonumber
\end{equation}

 For each $j$ and $\delta$, we consider the following equation
\begin{equation}\label{conforma3}
\left\{
\begin{aligned}
\Delta_{g_{\delta}}v_{\delta,j}-\frac{1}{8}f_{\delta, j}v_{\delta,j}&=0\ \ \,\text{in}\ M_1\\
v_{\delta,j}&=1\ \ \,\text{on}\ \Sigma_1\\
v_{\delta,j}&\rightarrow 1\ \ \text{at}\ \infty.
\end{aligned}
\right.
\end{equation}
By the result in \cite{SY} , the above equation has a $C^2$ positive solution
\begin{equation}\label{ asymptotic expansion}
v_{\delta,j}(x)=1+\frac{C_{\delta,j}}{|x|}
+\frac{O(1)}{|x|^2},
\end{equation}
 where $|O(1)|$ can be bounded by a uniform constant independent of $j$ and $\delta$.
We have the following estimates for $\{v_{\delta,j}\}$.

\begin{lm} \label{L6}
There exist uniform constants $j_0, \delta_1$ and $C$ which is independent of $j$  such that for any $j\geq j_0$
 and $\delta\in(0,\delta_1)$, we have
\[
\|v_{\delta,j}-1\|_{L^6}\leq C,\ \
\|\nabla_{g_{\delta}} v_{\delta,j}\|_{L^2}\leq C,\ \
\limsup_{\delta\rightarrow0}\|v_{\delta,j}-1\|_{L^6}\leq \frac{C}{2^j}.
\]
\end{lm}

\begin{proof}
Set $w_{\delta,j}=v_{\delta,j}-1$. Then $w_{\delta,j}$ satisfies
\begin{equation}\label{equa for w}
\left\{
\begin{aligned}
\Delta_{g_{\delta}}w_{\delta,j}-\frac{1}{8}f_{\delta, j}w_{\delta,j}&=\frac{1}{8}f_{\delta, j}\ \ \,\text{in}\ M_1\\
w_{\delta,j}&=0\ \ \ \ \ \ \ \, \text{on}\ \Sigma_1\\
w_{\delta,j}&\rightarrow 0 \ \ \ \ \ \ \ \text{at}\ \infty.
\end{aligned}
\right.
\end{equation}
It follows that
\[
\int_{M}(w_{\delta,j}\Delta_{g_{\delta}}w_{\delta,j}-\frac{1}{8}f_{\delta, j}w_{\delta,j}^2)dg_{\delta}
=\int_{M}\frac{1}{8}f_{\delta, j}w_{\delta,j}dg_{\delta}.
\]
Integrating by parts and using Holder Inequality, we have
\begin{align}
\int_{M_1}|\nabla_{g_{\delta}} w_{\delta,j}|^2dg_{\delta}
\leq& C
\Big(\int_{M_1}|f_{\delta, j}|^{\frac{3}{2}}\Big)^{\frac{2}{3}}
\Big(\int_{M_1}w_{\delta,j}^{6}dg_{\delta}\Big)^{\frac{1}{3}}\nonumber\\
&+C\Big(\int_{M_1}|f_{\delta, j}|^{\frac{6}{5}}dg_{\delta}\Big)^{\frac{5}{6}}
\Big(\int_{M_1}w_{\delta,j}^6dg_{\delta}\Big)^{\frac{1}{6}}\label{IBP}.
\end{align}
On the other hand, the Sobolev inequality gives
\begin{equation}\label{Sobolev ineq}
\Big(\int_{M_1}w_{\delta,j}^6dg_{\delta}\Big)^{\frac{1}{3}}
\leq C_{\delta}\int_{M_1}|\nabla_{g_{\delta}} w_{\delta,j}|^2dg_{\delta}.
\end{equation}
Since $g_{\delta}$ uniformly converge to $g$, then $C_{\delta}$
 is uniformly close to the Sobolev constant of $g$.
Combining (\ref{IBP}) with (\ref{Sobolev ineq}) and using Cauthy-Schwarz inequality give
\begin{align}
\Big(\int_{M_1}w_{\delta,j}^6dg_{\delta}\Big)^{\frac{1}{3}}
\leq &C\Big(\int_{M_1}|f_{\delta, j}|^{\frac{3}{2}}\Big)^{\frac{2}{3}}
\Big(\int_{M_1}w_{\delta,j}^{6}dg_{\delta}\Big)^{\frac{1}{3}}\nonumber\\
&+C\Big(\int_{M_1}|f_{\delta, j}|^{\frac{6}{5}}dg_{\delta}\Big)^{\frac{5}{3}}
+\frac{1}{2}\Big(\int_{M_1}w_{\delta,j}^6dg_{\delta}\Big)^{\frac{1}{3}}.
\end{align}
Thus, we can find  fixed $j_0$ and $\delta_1$,
such that for any $j\geq j_0$ and $\delta\in(0,\delta_1)$
we have
\begin{equation}\label{l6}
\Big(\int_{M_1}w_{\delta,j}^6dg_{\delta}\Big)^{\frac{1}{3}}
\leq C\Big(\int_{M_1}|f_{\delta, j}|^{\frac{6}{5}}dg_{\delta}\Big)^{\frac{5}{3}}.
\end{equation}
Combining (\ref{IBP}) and (\ref{l6}) gives the desired estimates.
\end{proof}

\section{Proof of main results}
In this section, we  prove  main theorems  by using results established in previous sections.

\begin{thm}\label{H_-=H_+}
Let $g$ be an asymptotically flat  metric satisfying the assumption as Theorem \ref{rigidity}.
Then
\[
H_-(z)=H_+(z)\ \ \text{for all}\ \, z\in\SO.
\]
\end{thm}
\begin{proof}

 Suppose that
\[
H_{-}(p)-H_{+}(p)>0\ \ \,\text{for some}\  p\in\SO.
\]
Choose a fixed leaf $\Sigma_1 \in \mathcal{F}$ that is very close to $\SH$ with respect to $g_-$. As mentioned above, $H(\Sigma_1)=H_1>0$ and $\Sigma_1$ is outer minimizing in $(M, g)$.

Now we fix some $r\leq\varepsilon_2$ and choose a sequence of $\delta_i\rightarrow 0$. Let $u_{\delta_i}$ and $v_{\delta_{i} ,r}$ be the solutions of Equations (\ref{conformal equation1}) and (\ref{conformal equation2}) respectively with $\delta =\delta_i$, and $\hat{g}_{\delta_i, r}$ be the resulting metrics.
Then by Lemma \ref{deform1} and Lemma \ref{minimizing hull},
for $i$ large enough, $\Sigma_1$ is outer minimizing in $(M, \hat{g}_{\delta_i, r})$
and has positive mean curvature. Note that $\frac{\partial}{\partial \nu}v_{\delta_{i} ,r}\leq 0$ on $\Sigma_1$.
 Then we consider the following two separate cases.

\textbf{Case 1.}\,
$\limsup_{i\rightarrow\infty}\int_{\Sigma_1}\frac{\partial}{\partial \nu}v_{\delta_{i} ,r}<0$.

In this case, using the result of \cite{HI}, we have
\begin{align}
\ADMm(M, \hat{g}_{\delta_i, r})\geq& \Hm(\Sigma_1, \hat{g}_{\delta_i, r})\nonumber\\
=&\sqrt{\frac{\Area(\Sigma_1)}{16\pi}}\left(1-\frac{1}{16\pi}\int_{\Sigma_1}H^2\left(\Sigma_1, \hat{g}_{\delta_i, r}\right)\right).\label{adm-hawking}
\end{align}
A direct computation shows
\[
H(\Sigma_1, \hat{g}_{\delta_i, r})=H_1+4\frac{\partial u_{\delta_i}}{\partial \nu}
+4\frac{\partial v_{\delta_i,r}}{\partial \nu}.
\]
Note that $H(\Sigma_1, \hat{g}_{\delta_i, r})>0$ and
$\frac{\partial}{\partial \nu}u_{\delta_i}\rightarrow 0$ as $i\rightarrow\infty$.
Taking upper limits of the both sides of (\ref{adm-hawking}) and note that the assumption  $\lim\sup_{i\rightarrow\infty}\int_{\Sigma_1}\frac{\partial}{\partial \nu}v_{\delta_{i} ,r}<0$,
we   obtain
\[
\limsup_{i\rightarrow\infty}\ADMm(M, \hat{g}_{\delta_i, r})>\Hm(\Sigma_1, g)\geq \ADMm(M, g).
\]
On the other hand, using $v_{\delta_i, r}\leq 1$ in $M$, we have
\[
\limsup_{i\rightarrow\infty}\ADMm(M, \hat{g}_{\delta_i, r})\leq
 \limsup_{i\rightarrow\infty}\ADMm(M, \tilde{g}_{\delta_i})
 =\ADMm(M, g).
\]
Thus, we get the desired contradiction.

\textbf{Case 2.}\,
$\limsup_{i\rightarrow\infty}\int_{\Sigma_1}\frac{\partial}{\partial \nu}v_{\delta_{i} ,r}=0$.

By the maximum principle, we have $v_{\delta_{i} ,r}\leq 1$ outside $\Sigma_1$. Since $v_{\delta_{i} ,r}\equiv 1$ on $\Sigma_1$, we get $\frac{\partial}{\partial \nu}v_{\delta_{i} ,r}\leq 0$. Hence, in this case we actually have $$\limsup_{i\rightarrow\infty}\frac{\partial v_{\delta_{i} ,r}}{\partial \nu}=0\,\,\,\,\,\text{on}\,\,\Sigma_1.$$
Then we obtain
\begin{equation}\label{estimate3}
\limsup_{i\rightarrow\infty}\ADMm(M, \hat{g}_{\delta_i, r})\geq \limsup_{i\rightarrow\infty}\Hm(\Sigma_1,\hat{g}_{\delta_i, r})\geq \mathfrak{m}.
\end{equation}
On the other hand, we have the following relation.
\begin{equation}
\begin{split}
\ADMm(M, \tilde{g}_{\delta_i})=&\ADMm(M, \hat{g}_{\delta_i, r})+\frac{1}{2\pi} \lim_{\rho\rightarrow\infty}\int_{S_\rho} \frac{\partial v_{\delta_i, r}}{\partial \nu}\\
=&\ADMm(M, \hat{g}_{\delta_i, r})
+\frac{1}{2\pi}\int_M \psi_{\delta_i,r} v_{\delta_i, r}   dV_{\tilde{g}_{\delta_i}}+\frac{1}{2\pi}\int_{\Sigma_1} \frac{\partial v_{\delta_i, r}}{\partial \nu}\\
=&\ADMm(M, \hat{g}_{\delta_i, r})
+\frac{1}{2\pi}\int_M \psi_{\delta_i,r} (v_{\delta_i, r}-1)   dV_{\tilde{g}_{\delta_i}}+\frac{1}{2\pi}\int_{\Sigma_1} \frac{\partial v_{\delta_i, r}}{\partial \nu}\\
&+\frac{1}{2\pi}\int_M \psi_{\delta_i,r} dV_{\tilde{g}_{\delta_i}}.
\end{split}
\end{equation}
Note that $0<v_{\delta_i, r} \leq 1$,  together with Lemma \ref{estimate2}   and (\ref{estimate3}), we get
\begin{equation}
\begin{split}
\mathfrak{m}&\geq\limsup_{i\rightarrow\infty}\ADMm(M, \hat{g}_{\delta_i, r})
+\limsup_{i\rightarrow\infty}\frac{1}{4\pi}\int_M\psi_{\delta_i,r}dV_{\tilde{g}_{\delta_i}}\\
&\geq\mathfrak{m}+\limsup_{i\rightarrow\infty}\frac{1}{4\pi}\int_M\psi_{\delta_i,r}dV_{\tilde{g}_{\delta_i}}.\end{split}
\end{equation}
By Lemma \ref{estimate2},
$$
\limsup_{i\rightarrow\infty}\frac{1}{4\pi}\int_M\psi_{\delta_i,r}dV_{\tilde{g}_{\delta_i}}
\geq\frac{\varepsilon_0 r^2}{8\pi}>0.
$$
Hence, we get the contradiction and finish the proof.
\end{proof}
Next, we are going to show $(\Omega, g_-)$ is static.

\begin{thm}\label{static}
Let $g$ be a  metric satisfying the assumption as Theorem \ref{rigidity}. Then
$(\Omega, g_-)$ must be static with $R(g_-)\equiv0$ in $\Omega$.
\end{thm}

\begin{proof}
We only give the proof that $(\Omega, g_-)$ is static since the vanishing of the scalar curvature can be shown in the same way.
We take the contradiction argument.
\begin{cla}Let $\mathcal{F}$ be the foliation of stable CMC surfaces constructed in Proposition \ref{foliation}, $\Sigma_1$ be a leaf in $\mathcal{F}$ and close enough to $\SH$,
then the domain $(\Omega_1, g_-)$ enclosed by $\Sigma_1$ and $\SO$ is static.
\end{cla}

Suppose  not, then by Theorem 1 in \cite{Cor}, there must exist some  perturbation of $g_-$, denoted by $\bar{g}$, such that

$\bullet$ $ \|\bar{g}-g_-\|_{C^2(\bar{\Omega}_1)}\leq \varepsilon$ with $\bar{g}|_{\partial\Omega_1}=g|_{\partial\Omega_1}$.

$\bullet$ $ H(\Sigma_1, \bar{g})=H(\Sigma_1, g_-)$ and $H(\SO, \bar{g})=H(\SO, g_-)$.

$\bullet$ $ R(\bar{g})\geq 0$ with $R(\bar{g})>0$ at some point $q\in \Omega_1$.\\
Here $\varepsilon>0$ is a constant small enough. Without loss of  generality,  we assume that $q\notin \mathcal{F}$. We  focus on the AF manifold $(\Omega_+, g_+)\cup (\bar{\Omega}_1, \bar{g})$. For simplicity, we  denote this manifold by $(M_1,g)$, where
\begin{equation}
g=
\left\{
\begin{aligned}
&g_+\,\,\,\,\,\text{for}\,\,x\in M_1\setminus\Omega,\\
&\bar{g}\,\,\,\,\,\,\,\,\,\text{for}\,\,x\in\Omega_1.
\end{aligned}
\right.
\end{equation}
Then $g$ is  a metric admitting corners along $\SO$.
Furthermore, by Remark \ref{outerminimizing in small perturbation}, we may choose $\varepsilon$ small enough such that
$\Sigma_1$ is still outer minimizing in $(M_1, g)$.
Using the same construction as in the previous section, we get a family of $C^2$ metrics $g_{\delta}$.
Now we fix a $C^1$ vector field $X$ near $\SO$ such that $X(y)=\nu_{\bar{g}}(y)$, for any $ y\in\SO$.

Let $v_{\delta,j}$ be the solution of (\ref{conforma3}). By Lemma \ref{L6}  and  the standard  elliptic estimates, we have
$$\|v_{\delta,j}\|_{W^{2,6}(K)}\leq C_K\,\,\,\, \text{for any} ~K\subset\subset\bar{M_1}.$$
By Theorem 7.26 in \cite{GT},
$\|v_{\delta,j}\|_{C^{1,\alpha}(K)}\leq C_K$ for some $\alpha\in(0,1)$.
Furthermore, we have
\begin{equation}\label{1alpha}
\lim_{\delta\rightarrow 0}\|v_{\delta,j}-1\|_{C^{1,\alpha}(K)}\leq \theta_{K,j}
\end{equation}
with $\theta_{K,j}\rightarrow 0$ as $j\rightarrow \infty$.
Set $V_{\delta_0}=\SO\times(-\delta_0,\delta_0)$, then $\|v_{\delta,j}\|_{C^{1,\alpha}(V_{\delta_0})}$
is uniformly bounded.
On the other hand,
take any compact set $K\subset\bar{M}_1\setminus \SO$,
then the Schauder estimate implies that there exists some constant $C_K$ independent of $j$ and $\delta$ such that
\begin{equation}\label{schauder estimate}
\|v_{\delta,j}\|_{C^{2,\alpha}(K)}\leq C_K.
\end{equation}
Thus, there
exist $\delta_i\rightarrow 0$ such that
$v_{\delta_i,j}\rightarrow v_j$ in the $C^2_{loc}$ sense and
$v_{\delta_i,j}$ strongly converge to $v_j$ in the $C^{1,\frac{\alpha}{2}}(V_{\delta_0})$ sense.
 Moreover, $v_j$ satisfies
\begin{equation}
\left\{
\begin{aligned}
\Delta_g v_j-\frac{1}{8} f_j v_j&=0\ \,\,\text{in}\  M_1,\\
v_j&=1\ \,\,\text{on}\ \Sigma_1,\\
v_j&\rightarrow 1\ \,\,\text{at}\ \infty.
\end{aligned}
\right.
\end{equation}
Here $f_j\geq 0$ and Supp$(f_j)\subset \subset\Omega_1$.
By maximum principle, $v_j(x)\leq 1$.
Furthermore, we have

\begin{cla}
$v_j(x)< 1$,  for any  $x\in M\setminus \bar{\Omega}$.
\end{cla}
Suppose that $v_j(y)=1$ for some $y\in M\setminus \bar{\Omega}$, then we have
\begin{equation}
\left\{
\begin{aligned}
\Delta_gv_j&=0\ \,\,\,\text{in}\  M\setminus \bar{\Omega},\\
v_j&\leq 1\ \ \,\text{on}\  \SO,\\
v_{ j}&\rightarrow1 \ \,\,\text{at}\ \infty.
\end{aligned}
\right.
\end{equation}
Then the strong maximum principle implies that $v_j\equiv1 $ in $M\setminus \bar{\Omega}$.
It follows that $X(v_j)\equiv 0$ on $\SO$.
On the other hand, we have
\begin{equation}
\left\{
\begin{aligned}
\Delta_gv_j-\frac{1}{8}f_jv_j&=0 \ \,\,\, \text{in}\ \Omega_1,\\
v_j&=1\ \,\,\, \text{on}\ \partial\Omega_1.
\end{aligned}
\right.
\end{equation}
Since $v_j\leq 1$  and $f_j>0$ at some point $p\in \Omega_1$, by Hopf's Lemma,
\[
\frac{\partial}{\partial v_j}\nu_{\bar{g}}=X(v_j)>0 \ \,\,\, \text{on}\ \SO.
\]
 Note that $v_j$ is
$C^{1,\alpha}$ in a neighbourhood of $\SO$.
Thus we get the desired contradiction.

The similar argument shows that
$\frac{\partial}{\partial \nu_g}v_j<0$ on $\Sigma_1$.
Thus, for $i$ large enough, we have
$\frac{\partial}{\partial \nu_g}v_{\delta_i,j}<0$ on $\Sigma_1$.
Take a fixed coordinate sphere  $S_{R}$ with $R\gg1$,
then for $i$ large enough, we have
$v_{\delta_i, j}<1$ on $S_{R}$.
Note each $v_{\delta_i, j}$ is a harmonic function in $\{x:|x|\geq R\}$.
It follows $v_{\delta_i, j}<1$ in $\{x:|x|\geq R\}$.
By (\ref{ asymptotic expansion}), we have $C_{\delta_i,j}< 0$.

Set $\tilde{g}_{\delta_i,j}=v_{\delta_i,j}^4g_{\delta_i}$. Then
\begin{align*}
R(\tilde{g}_{\delta,j})=&v_{\delta,j}^{-5}(R(g_{\delta_i})v_{\delta_i,j}-8\Delta_{g_{\delta_i}}v_{\delta_i,j})\\
=&v_{\delta_i,j}^{-5}(R(g_{\delta_i})v_{\delta_i,j}-f_{\delta_i, j}v_{\delta_i,j})\\
\geq&0.
\end{align*}
The ADM mass of $(M,\tilde{g}_{\delta_i,j})$ is given by
\[
\ADMm(M,\tilde{g}_{\delta_i, j})=\mathfrak{m}+2C_{\delta_i, j}<\mathfrak{m}.
\]

Choose a sequence of $\delta_i\rightarrow0$. By (\ref{1alpha}) and (\ref{schauder estimate})
and Remark \ref{outerminimizing in small perturbation},
there exist some $j\gg1$ such that
both $\|v_{\delta_i, j}-1\|_{C^1(B_R)}$ and $\|v_{\delta_i, j}-1\|_{C^2(B_R\setminus B_{R/2})}$
are small enough to guarantee that
for $i$ large enough, $\Sigma_1$ is outer minimizing in $(M_1, \tilde{g}_{\delta_i,j})$ and
$H(\Sigma_1, \tilde{g}_{\delta_i,j})>0$.
Consider the inverse mean curvature flow with
 $\Sigma_1\subseteq (M, \tilde{g}_{\delta_i,j})$ as the initial data.
The result of Huisken-Ilmanen in \cite{HI} implies that
\[
\Hm(\Sigma_1,\tilde{g}_{\delta_i,j})\leq \ADMm(M, \tilde{g}_{\delta_i,j}).
\]
Recall that the Hawking mass of $\Sigma_1$ is given by
\[\Hm(\Sigma_1, \tilde{g}_{\delta_i,j})=\sqrt{\frac{\Area(\Sigma_1)}{16\pi}}
\left(1-\frac{1}{16\pi}\int_{\Sigma_1}H^2\left(\Sigma_1,\tilde{g}_{\delta_i,j}\right)\right)
\]
Note that
\[
H(\Sigma_1,\tilde{g}_{\delta_i,j})=H(\Sigma_1, \bar{g})+4\frac{\partial v_{\delta_i,j}}{\partial \nu_g}.
\]
Thus, letting $i\rightarrow \infty$ gives
\[
\Hm(\Sigma_1,\bar{g})<\mathfrak{m}.
\]
Hence, we see that $(\Omega_1, g_-)$ is static. 

Generally, let $\{\Sigma_i\}$ be a sequence of leaf in $\mathcal{F}$ that approach $\SH$. By the same argument, we can prove $(\Omega_i, g_- |_{\Omega_i})$ is static. According to the definition of static metrics, there are corresponding functions $\{f_i\}$, called the potential functions, satisfying 
\begin{equation}\label{staticequation}
\left\{
\begin{aligned}
\nabla^2 f_i&= f_i \ric(g_-) \ \ \text{in}\,\,\,\Omega_i\\
\Delta f_i&=0 \ \ \ \ \ \ \qquad\,\,\, \text{in}\,\,\,  \Omega_i,\end{aligned}
\right.
\end{equation}
By Proposition 2.5 in \cite{Cor}, $f_i$ is $C^2$ smooth up to the boundary of $\Omega_i$. Without loss of  generality, we may assume there is a $p_i$ in $\bar \Omega_i$ with
$$
f_i (p_i)=\sup_{\Omega_i}|f_i|=1.
$$
Then \eqref{staticequation} implies that there are constants $\Lambda_k$ depending only on $g_-$ and $k$ such that
\begin{equation}\label{estimate4}
\| f_i\|_{C^k(\bar \Omega_i)}\leq \Lambda_k.
\end{equation}

Assume $\{p_i\}$ converges to $p$ which is at the boundary of $\Omega$, together with (\ref{estimate4}), we see that there is a fixed  constant $\theta_0$ with
$$
f_i(x)\geq \frac12,
$$
for any $x\in B_{\theta_0} (p_i)\cap \Omega_i$. Let $f$ be the limit of $\{f_i\}$, then we see that it satisfies (\ref{staticequation}) in $\Omega$,  and $f \in C^{2}(\bar \Omega)$
and is positive in $B_{\theta_0} (p)\cap \Omega$. Thus $(\Omega, g_-)$ is static and $f$ is its nontrivial potential function. Thus we finish prove the Theorem.
\end{proof}

Finally, we want to show if the ADM mass of $(M, g)$ is equal to $\Hm(\SH, g_-)$ then $(\Omega,g_- )$ can be isometrically embedded into a Schwarzschild manifold.
First, we note that
\begin{pro}\label{ localSchwar}
There exists a neighbourhood $U$ near $\SH$ in which $g_-$ is the Schwarzschild metric.
\end{pro}
\begin{proof}
Take any leaf $\Sigma'\in\mathcal{F}$ and
 consider the inverse mean curvature flow $\{X_t\}_{t\in[0,T)}$ with $X_0=\Sigma'$ as the initial data.
 Here $T$ is chosen such that every slice $X_s\in W_0, s\in[0, T)$.
 Then  $X_s$ is strictly outer minimizing  for $s\in(0, \frac{T}{2})$.
Using the similar argument as in Theorem \ref{static}, we can show
 \[
 \Hm\left(X_s\right)=\Hm\left(\SH\right), s\in\left(0, T/2\right).
 \]
It follows from  the rigidity result proved by  Huisken-Ilmenen (See P.423-424, \cite{HI}) that $g_-$ is isometric to $g_{\scriptscriptstyle{S}}$ in the domain enclosed by
 $X_{T/2}$ and $\Sigma'$. Since $\Sigma'$ can be arbitrary close to $\SH$. Hence, we complete the proof.
\end{proof}

Let $f_-$ be a potential function of $g_-$ on $\Omega$, $f_{\scriptscriptstyle{S}}$ be the standard potential function on $U$ where $g_-$ is a Schwarzschild metric. Then

\begin{lm}\label{potentialfunctions}
There is a nonzero constant $\lambda$ with
$$
f_-=\lambda f_{\scriptscriptstyle{S}}\ \ \,\text{in}\ U.$$
\end{lm}

\begin{re}
Without loss of generality, we assume $\lambda=1$ in the sequel.
\end{re}

\begin{proof}
It suffices to show that there is a nonzero constant $\lambda$ with function $f_\lambda =f_{\scriptscriptstyle{S}} -\lambda f_-$ vanishes at a point in  $\SH$ up to the first order.
Indeed, let   $K=f_{\scriptscriptstyle{S}}\nabla f_- -f_-\nabla f_{\scriptscriptstyle{S}}$.  Then  for any vector $X$, we have
\begin{align*}
\nabla_XK=&\nabla_X\left(f_{\scriptscriptstyle{S}}\nabla f_- -f_-\nabla f_{\scriptscriptstyle{S}}\right)\\
=&\left(Xf_{\scriptscriptstyle{S}}\right)\nabla f_- +f_{\scriptscriptstyle{S}}\nabla^2f_ -\left(X,\cdot\right)-\left(Xf_-\right)\nabla f_{\scriptscriptstyle{S}}-f_-\nabla^2f_{\scriptscriptstyle{S}}\left(X,\cdot\right)\\
=&\left(Xf_{\scriptscriptstyle{S}}\right)\nabla f_- -\left(Xf_-\right)\nabla f_{\scriptscriptstyle{S}}.
\end{align*}
Accordingly, for any vector $Y,Z$, we have
\begin{align*}
\langle\nabla_YK,Z\rangle+\langle\nabla_ZK,Y\rangle=0.
\end{align*}
So $f_{\scriptscriptstyle{S}}\nabla f_- -f_-\nabla f_{\scriptscriptstyle{S}}$ is Killing. Note that $g_-$ is Schwarzschild near $\SH$, we have $K^{\perp}=0$ on $\SH$, together with fact that
$f_{\scriptscriptstyle{S}}=0$ on $\SH$, we get $f_-=0$ on $\SH$. Take any point $p$ in $\SH$, let $\lambda$ be the constant with $\nabla f_{\scriptscriptstyle{S}} (p)=\lambda \nabla f_-(p)$,
then $\lambda\neq 0$ and we see that
$f_\lambda =f_{\scriptscriptstyle{S}} -\lambda f_-$ vanishes at $p \in \SH$ up to the first order.
 By ODE argument, we know that $f_\lambda =f_{\scriptscriptstyle{S}} -\lambda f_-\equiv0$ in $U$. Hence, we finish the proof the Lemma.
\end{proof}

For any hypersurface $\Sigma \subset U$, if  its Gauss curvature of the induced metric from $g_-$ is a positive constant, then we may find a smooth function $\phi(\rho)$ so that the metric $g_2$ defined by
\[g_2 = d\rho^2 + \phi^2(\rho) d\omega^2\]
 is the Schwarzschild metric with ADM mass $\mathfrak{m}$, here $\mathfrak{m}=\Hm(\SH, g_-)$, $\rho$ is the distance function to $\Sigma$ with respect to $g_-$, $d\omega^2$ is the standard metric on the $\mathbb{ S}^2$. For simplicity, we use $g_1$ to denote metric $g_-$.
  Without loss of  generality, we may assume  $g_1$, $g_2$ are defined on $U=[0, 1]\times \Sigma $, then they are static metrics on $U$, let  $f_i$,  $i=1$, $2$, be their potential functions, and they coincide on $U_\epsilon=[0, \epsilon]\times \Sigma $, here $\epsilon>0$,  where $f_1=f_2 >0$, we want to show  $g_1$ and $g_2$ is isometric, and $f_1=f_2$ on the whole $U$.
For any point $p \in U$, we denote it as $p=(\rho, \theta)$, here $\rho \in [0, 1]$,
 $\theta=(\theta^1, \dots, \theta^n)$ is the local coordinates on $\Sigma$. Under this coordinates, for $i=1$, $2$, we assume
\begin{equation}
g_i =d\rho^2 +( g_i )_{mn} (\rho, \theta)d\theta^m d\theta^n ,  \nonumber
\end{equation}
and
\begin{equation}\label{condition1}
\begin{split}
&( g_1 )_{mn} (\rho, \theta) = ( g_2 )_{mn} (\rho, \theta), \\
&f_1(\rho, \theta)=f_2(\rho, \theta), \quad \text { for any $\rho \in [0, \epsilon]$, $\theta \in \Sigma$.}
\end{split}
\end{equation}

Then by the  same arguments in the proof of Theorem 1.1 in \cite{CD} (see also the arguments in the proof of Theorem 1 in \cite{Biq}), we have
\begin{pro}\label{uniquecontinuation}
Suppose $(g_i, f_i)$ , $i=1$, $2$, are two smooth and static metrics on $U$, and satisfy (\ref{condition1}), then there is a constant $\epsilon<\delta_0 \leq 1$ which is independent of $\epsilon$ with
\begin{equation}
\begin{split}
&( g_1 )_{mn} (\rho, \theta) = ( g_2 )_{mn} (\rho, \theta), \\
&f_1(\rho, \theta)=f_2(\rho, \theta), \quad \text { for any $\rho \in [0, \delta_0]$, $\theta \in \Sigma$.}
\end{split}
\end{equation}
\end{pro}

\begin{proof}[Proof of Theorem \ref{rigidity}]
By the assumption of Theorem \ref{rigidity}, we see that for any $p \in \Omega$, it is contained in a Fermi coordinates $\mathcal{O}_p =[0, a)\times W_p$ for some constant $a>0$, here $W_p$ is an open subset of $\SH$ which contains the nearest  point projection of $p$ in $\SH$. In the coordinates $\mathcal{O}_p$, $g_1= g_-= d\rho^2 +( g_1)_{mn} (\rho, \theta)d\theta^m d\theta^n$, and it is Schwarzschild for $\rho<\epsilon$. By Proposition \ref{uniquecontinuation}, we see that $g_-$ is Schwarzschild on $\mathcal{O}_p$.  Due to
Proposition \ref{ localSchwar}, we see that there is isometric $\Psi$: $(\SH, g_-|_{\SH})$: $\mapsto (\partial \mathbb{X}_{\mathfrak{m}}, g_{\scriptscriptstyle{S}}|_{\partial\mathbb{X}_{\mathfrak{m}}})$. Note that there is a global Fermi coordinates on  $(\mathbb{X}_{\mathfrak{m}}, g_{\scriptscriptstyle{S}} )$.
 Hence,  any $q\in \mathbb{X}_{\mathfrak{m}}$ can be expressed as $q=(r, \theta)$ where $r$ is the distance of $q$ to $\partial \mathbb{X}_{\mathfrak{m}}$ with respect to the Schwarzschild metric $g_{\scriptscriptstyle{S}}$, and $\theta \in\partial \mathbb{X}_{\mathfrak{m}}$ is the nearest point projection of $q$ to $\partial \mathbb{X}_{\mathfrak{m}}$ with respect to the Schwarzschild metric $g_{\scriptscriptstyle{S}}$. Then we define
$$
F_p : \mathcal{O}_p \mapsto \mathbb{X}_{\mathfrak{m}}
$$
by
$$
F_p(x)=(\rho(x), \Psi(\bar x)),
$$
for any $x\in \mathcal{O}_p$, and here $\bar x$ is the nearest point projection of $x$ to $\SH$ with respect to metric $g_-$.  By above discussion, we see that $F_p$ is isometric, and for any $p$, $q$ in $\Omega$, we have
$$
F_p |_{\mathcal{O}_p \cap \mathcal{O}_q}= F_q |_{\mathcal{O}_p \cap \mathcal{O}_q}.
$$
Thus, we can extend $F_p$ to an isometric map $F$: $(\Omega, g_-) \mapsto (\mathbb{X}_{\mathfrak{m}}, g_{\scriptscriptstyle{S}})$.
 Hence, we finish the proof.
\end{proof}

\end{document}